\documentclass[12pt,reqno]{amsart}
\usepackage{amssymb,amscd,amsthm,latexsym}

\usepackage[all]{xy}
\usepackage{color}
\usepackage{hyperref}
\hypersetup{
    colorlinks,
    citecolor=red,
    filecolor=red,
    linkcolor=red,
    urlcolor=red
}

\theoremstyle{plain}
\newtheorem{theorem}[subsection]{Theorem}

\newtheorem{proposition}[subsection]{Proposition}
\newtheorem{corollary}[subsection]{Corollary}

\theoremstyle{definition}

\newtheorem{example}[subsection]{Example}

\newtheorem{remark}[subsection]{Remark}

\newdir{ >}{% end of arrow for the monos
	@{}*!/-10pt/@{>} }

\newcommand{\C}{\ensuremath{\mathbb{C}}}
\newcommand{\V}{\ensuremath{\mathbb{V}}}

\newcommand{\Eq}{ \ensuremath{\mathrm{Eq}} }

\newcommand{\map}[2]{ \ensuremath{ \xymatrix@1@C=15pt{ #1 \ar[r] & #2 } } }
\newcommand{\mono}[2]{ \ensuremath{ \xymatrix@1@C=15pt{ #1\; \ar@{ >->}[r] & #2 } } }
\newcommand{\regepi}[2]{ \ensuremath{ \xymatrix@1@C=15pt{ #1 \ar@{>>}[r] & #2 } } }

\def\pullback{% with thanks to Valerian Even
	\ar@{-}[]+R+<6pt,-1pt>;[]+RD+<6pt,-6pt>%
	\ar@{-}[]+D+<1pt,-6pt>;[]+RD+<6pt,-6pt>}

\begin{document}
\title[Variations of the Shifting Lemma and Goursat categories]%
{Variations of the Shifting Lemma and Goursat categories}

\author{ Marino Gran}
\address{Institut de Recherche en Math\'ematique et Physique, Universit\'e Catholique de Louvain, Chemin du Cyclotron 2,
	1348 Louvain-la-Neuve, Belgium}

\email{marino.gran@uclouvain.be}

\author[Diana Rodelo]{Diana Rodelo}
\address{Departamento de Matem\'atica, Faculdade de Ci\^{e}ncias e Tecnologia, Universidade
	do Algarve, Campus de Gambelas, 8005-139 Faro, Portugal and CMUC, Department of Mathematics, University of Coimbra, 3001-501 Coimbra, Portugal}
\thanks{The second author acknowledges partial financial assistance by Centro de Matem\'{a}tica da
	Universidade de Coimbra---UID/MAT/00324/2013, funded by the
	Portuguese Government through FCT/MCTES and co-funded by the
	European Regional Development Fund through the Partnership
	Agreement PT2020.}
\email{drodelo@ualg.pt}

\author{ Idriss Tchoffo Nguefeu}
\address{Institut de Recherche en Math\'ematique et Physique, Universit\'e Catholique de Louvain, Chemin du Cyclotron 2,
	1348 Louvain-la-Neuve, Belgium}
\thanks{The third author acknowledges financial assistance by Fonds de la Recherche Scientifique-FNRS Cr\'edit Bref S\'ejour \`a l'\'etranger 2018/V 3/5/033 - IB/JN - 11440.}
\email{idriss.tchoffo@uclouvain.be}


 % Exact categories, abelian categories

%%%%%%%%%%%%%%%%%%%%%%%%%%%%%%%%%%%%%%%  ABSTRACT  %%%%%%%%%%%%%%%%%%%%%%%%%%%%%%%%%%%%%%%%%%%%%%%%%%%%%%%%%%%%%%%%%%%%
\begin {abstract}
We prove that Mal'tsev and Goursat categories may be characterized through variations of the Shifting Lemma, that is classically expressed in terms of three congruences $R$, $S$ and $T$, and characterizes congruence modular varieties. We first show that a regular category $\C$ is a Mal'tsev category if and only if the Shifting Lemma holds for reflexive relations on the same object in $\C$.
Moreover, we prove that a regular category $\C$ is a Goursat category if and only if the Shifting Lemma holds for a reflexive relation $S$ and reflexive and positive relations $R$ and $T$ in $\C$. In particular this provides a new characterization of $2$-permutable and $3$-permutable varieties and quasi-varieties of universal algebras.
\end {abstract}

%	\infonum{18}{}

\date{\today}

\maketitle

{\section*{Introduction}}
%%%%%%%%%%%%%%%%%%%%%%%%%%%%%%%%%%%%%%%  INTRODUCTION  %%%%%%%%%%%%%%%%%%%%%%%%%%%%%%%%%%%%%%%%%%%%%%%%%%%%%%%%%%%%%%%%
For a variety $\mathbb V$ of universal algebras, Gumm's \emph{Shifting Lemma}~\cite{Gumm} is stated as follows. Given congruences $R,S$ and $T$ on the same algebra $X$ in $\mathbb V$ such that $R\wedge S\leqslant T$, whenever $x,y,u,v$ are elements in $X$ with $(x,y)\in R \wedge T$, $(x,u) \in S$, $(y,v)\in S$ and $(u,v)\in R$, it then follows that $(u,v) \in T$. We display this condition as 
\begin{equation}\label{SL}
\vcenter{\xymatrix@C=30pt{
		x \ar@{-}[r]^-S \ar@{-}[d]^-R \ar@(l,l)@{-}[d]_-T & u \ar@{-}[d]_-R \ar@(r,r)@{--}[d]^-T \\
		y \ar@{-}[r]_-S  & v. }}
\end{equation}

A variety $\mathbb V$ of universal algebras satisfies the Shifting Lemma precisely when it is congruence modular~\cite{Gumm}, this meaning that the lattice of congruences on any algebra in $\mathbb V$ is modular. In particular, since any $3$-permutable variety is congruence modular \cite{Jonsson53}, it always satisfies the Shifting Lemma. Recall that a variety $\V$ is \emph{$3$-permutable} when, given any congruences $R$ and $S$ on the same algebra $X$, we have the equality
\begin{equation}\label{3-p}
RSR=SRS,
\end{equation}
where
$$
RSR \!=\! \{ (x,u) \!\in\! X \times X \mid \exists_{y, z \in X}  \, ( (x,y) \in R \, \& \,  (y,z) \in S \, \& \,  (z,u) \in R) \}
$$
and
$$
SRS \!=\!  \{ (x,u) \!\in\! X \times X \mid  \exists_{y, z \in X}  \, ( (x,y) \in S \, \& \,  (y,z) \in R \, \& \,  (z,u))  \in S \}
$$
are the usual composites of congruences. Such varieties are characterized ~\cite{HM} by the existence of two ternary terms $r(x,y,z)$ and $s(x,y,z)$ satisfying the identities
$$
\begin{array}{l}
r(x,y,y)=x \\
r(x,x,y)=s(x,y,y) \\
s(x,x,y)=y.
\end{array}
$$

Any $2$-permutable variety \cite{Smith} is necessarily $3$-permutable, and there are known examples of varieties that are $3$-permutable, but fail to be $2$-permuta\-ble, such as the varieties of right complemented semigroups~\cite{HM} and the one of implication algebras~\cite{M}.

The notions of $2$-permutability and $3$-permutability can be extended from varieties to \emph{regular} categories by replacing \emph{congruences} with \emph{internal equivalence relations}, allowing one to explore some interesting new (non-varietal) examples. A categorical version of the Shifting Lemma (stated differently from the original formulation recalled above) may be considered in any finitely complete category, and this leads to the notion of a \emph{Gumm category}~\cite{BG0, BG1}. In a regular context one can use set-theoretic terms thanks to Barr's embedding theorem~\cite{Barr} (see also Metatheorem A.5.7 in~\cite{BB}), so that the property given in diagram~\eqref{SL} may still be expressed by using generalised elements. In a regular context one can show that the property of $3$-permutability still implies the modularity of the lattice of equivalence relations \cite{CKP}, and that this latter property implies that the Shifting Lemma holds. However, the converse is false even in the case of a variety of infinitary algebras, as it was shown by G. Janelidze in Example 12.5 in~\cite{GJ}.

Regular categories that are $2$-permutable, or $3$-permutable, are usually called \emph{Mal'tsev categories} \cite{CLP} and \emph{Goursat categories} \cite{CKP}, respectively. As examples of regular Mal'tsev categories that are not (finitary) varieties of algebras we list: $\mathrm{C}^*$-algebras, compact groups, topological groups \cite{CKP}, torsion-free abelian groups, reduced commutative rings, cocommutative Hopf algebras over a field \cite{GSV}, any abelian category, and the dual of any topos \cite{CKP}.

In the varietal context, H.-P. Gumm has also considered a slight variation of the Shifting Lemma called the \emph{Shifting Principle} \cite{Gumm}: given congruences $R$ and $T$ and a reflexive, symmetric and compatible relation $S$ on the same algebra $X$ such that $R\wedge S\leqslant T\leqslant R$, whenever $x,y,u,v$ are elements in $X$ with $(x,y)\in R \wedge T$, $(x,u) \in S$, $(y,v)\in S$ and $(u,v)\in R$ as in \eqref{SL}, it then follows that $(u,v) \in T$. The Shifting Principle, although apparently stronger, turns out to be equivalent to the Shifting Lemma in the varietal case.

With this observation in mind, it seems reasonable to expect that considering variations on the assumptions on the relations $R,S$ or $T$ appearing in the Shifting Lemma might provide characterizations of other types of categories. The variations we have in mind for $R,S$ and $T$ are to make those assumptions weaker, so that they give stronger versions of the Shifting Lemma. This idea comes from the well known characterization of Mal'tsev categories through the fact that reflexive relations are equivalence relations~\cite{CLP, CPP}, and from a more recent one of Goursat categories recalled in Proposition~\ref{Goursat positive}: a regular category is a Goursat category if and only if any reflexive and positive relation (i.e. a relation of the form $U^{\circ}U$, for some relation $U$) is an equivalence relation.

The main results of the present paper show that stronger versions of the Shifting Lemma characterize regular categories that are Mal'tsev categories (Theorems~\ref{Mal'tsev <=> SL} and \ref{Mal'tsev <=> SL 2}), and that are Goursat categories (Theorem~\ref{Goursat <=> SL}). These results apply in particular to $2$-permutable and $3$-permutable quasi-varieties, since these latter categories are known to be regular. \\

\section{Regular categories and relations}

A morphism in a category $\mathbb C$ is a regular epimorphism if it is a coequaliser of a pair of parallel morphisms in $\mathbb C$.
A category $\mathbb C$ with finite limits is called \emph{regular} if any morphism $f$ admits a (unique up to isomorphism) factorization $f=m r $, where $r$ is a regular epimorphism and $m$ is a monomorphism, and these factorizations are pullback stable. Since monomorphisms are always pullback stable, this latter property can be equivalently expressed by asking that in any pullback
$$
\xymatrix{E \times_B A \ar@{->>}[r]^-{\pi_2} \ar[d]_{\pi_1} \pullback & A \ar[d]^f \\
	E \ar@{->>}[r]_p & B
}
$$
the arrow $\pi_2$ is a regular epimorphism whenever so is $p$.
\begin{example}
	Any variety of universal algebras is a regular category, where regular epimorphisms are surjective homomorphisms, and finite limits (in particular, pullbacks) are computed as in the category of sets. The same is true for any quasi-variety of algebras (see \cite{PV}, for example). The factorization of any homomorphism $f \colon X \rightarrow Y$ as a regular epimorphism followed by a monomorphism is simply the factorization $ X \rightarrow f(X) \rightarrow Y$, where $f(X)$ is the direct image of $f$.
	A thorough study of quasi-varieties with modular lattice of congruences can be found in \cite{Kearnes}.
\end{example}
In this article, we mainly work in a regular category, thus the proofs are partially given in set-theoretic terms (see Metatheorem $A.5.7$ in \cite{BB}, for instance).
An (internal) \emph{relation} $R$ from  $X$ to  $Y$  is a subobject $\langle r_1,r_2 \rangle \colon \mono{R}{X \times Y}$. The opposite relation of $R$, denoted $R^{\circ}$, is the relation from $Y$ to $X$ given by the subobject $\langle r_2,r_1 \rangle\colon \mono{R}{Y \times X}$. A relation $R$ from $X$ to $X$ is called a relation on $X$. Given two relations $\langle r_1,r_2 \rangle \colon \mono{R}{X \times Y}$ and $\langle s_1,s_2 \rangle  \colon \mono{S}{Y \times Z}$ in a regular category, their relational composite $\mono{SR}{X \times Z}$ can be defined as follows: take the pullback 
$$\xymatrix{R \times_Y S \ar[r]^-{\pi_2} \ar[d]_{\pi_1} \pullback & S \ar[d]^{s_1}\\
	R \ar[r]_{r_2} & Y
}
$$
of $r_2$ and $s_1$, and the (regular epimorphism, monomorphism) factorization of $\langle r_1\pi_1,s_2\pi_2\rangle$:
$$
\xymatrix@R=5pt{R\times_Y S \ar[rr]^-{\langle r_1\pi_1,s_2\pi_2\rangle} \ar@{>>}[dr] & & X\times Z. \\
& SR \ar@{ >->}[ur]}
$$

A relation $\mono{R}{X\times X}$ on $X$ is called \emph{reflexive} when $1_X\leqslant R$, where $1_X$ denotes the relation $\langle 1_X,1_X\rangle \colon \mono{X}{X\times X}$. It is called \emph{symmetric} when $R^{\circ}\leqslant R$ (or, equivalently $R^\circ=R$), and \emph{transitive} when $RR\leqslant R$. A reflexive, symmetric and transitive relation is called an \emph{equivalence} relation (it then follows that $RR=R$). When the regular category is a variety $\mathbb V$, an equivalence relation in $\mathbb V$ is simply a congruence.

A relation $\mono{D}{X\times Y}$ is called \emph{difunctional} if $(x,y)\in D$
whenever $
(x,v)\in D, (u,v)\in D, (u,y)\in D$. In a finitely complete category this property can be expressed by the equality
$DD^\circ D=D$.
%$$
%\begin{array}{c}
%    xDv \\
%    uDv \\
%    uDy \\
%    \hline
%    xDy.
%\end{array}
%$$
%A relation $\mono{P}{X\times Y}$ is difunctional  using set-theoretic notations, if $(x,v)\in D$, $(u,v)\in D $, $(u,y)\in D$, then $(x,y)\in D$)
A relation $\mono{P}{X\times X}$ on $X$ is called \emph{positive} when it is of the form $P=U^\circ U$, for some relation $\mono{U}{X\times Y}$ \cite{Sel}. In set-theoretic terms, $P$ is positive when there exists a relation $U$ with $(x,x')\in P$ if and only if $(x,y)\in U$ and $(x',y)\in U$ for some $y$. It is easy to see that any positive relation is symmetric, and that any equivalence relation $R$ is positive since $R=R^\circ R$.

%%%%%%%%%%%%%%%%%%%%%%%%%%%%%%%%%%%%%%%%%%%%%%%%%%%%%%%%%%%%%%%%%%%%%%%%%%%%%%%%%%%%%%%%%%%%%%%%%%%%%%%%%%%%%%%%%%%%%%%
\section{Regular Mal'tsev categories and the Shifting Lemma}

A finitely complete category $\C$ is called a \emph{Mal'tsev category} if every reflexive relation in $\C$ is an equivalence relation~\cite{CLP,CPP}. These categories are also characterized by other properties on relations, as follows:

\begin{theorem}\label{Mal'tsev chars} \emph{\cite{CPP}}
	Let $\C$ be a regular category. The following conditions are equivalent:
	\begin{enumerate}
		\item[(i)] $\C$ is a Mal'tsev category;
		\item[(ii)] every relation $\mono{D}{X\times Y}$ in $\C$ is difunctional;
		\item[(iii)] every reflexive relation $E$ in $\C$ is symmetric: $E^\circ =E$.
	\end{enumerate}
\end{theorem}

Regular Mal'tsev categories may be characterized by the fact that they are $2$-permutable: for any pair of equivalence relations $R,S$ on the same object, $RS=SR$~\cite{CLP}. If $\C$ is a regular Mal'tsev category, then the lattice of equivalence relations on the same object is modular~\cite{CKP} and, consequently, the Shifting Lemma holds.

Using the fact that in a Mal'tsev category reflexive relations coincide with equivalence relations, or with symmetric relations, we are now going to show that regular Mal'tsev categories may be characterized through a stronger version of the Shifting Lemma where, in the assumption, the equivalence relations are replaced by reflexive relations. Note that, for a diagram such as \eqref{SL} where $R,S$ or $T$ are not equivalence relations, the relations are always to be considered from left to right and from top to bottom.

\begin{theorem}
	\label{Mal'tsev <=> SL} Let $\C$ be a regular category. The following conditions are equivalent:
	\begin{enumerate}
		\item[(i)] $\C$ is a Mal'tsev category;
		\item[(ii)] The Shifting Lemma holds in $\C$ when $R,S$ and $T$ are reflexive relations.
	\end{enumerate}
\end{theorem}
\begin{proof}
	(i) $\Rightarrow$ (ii) This implication follows from the fact that reflexive relations are necessarily equivalence relations and the Shifting Lemma holds in any regular Mal'tsev category. \\
	(ii) $\Rightarrow$ (i) We shall prove that every reflexive relation $E$ %$\langle e_1, e_2\rangle \colon \mono{E}{X\times X}$ 
is symmetric (which suffices by Theorem~\ref{Mal'tsev chars}(iii)). Suppose that $(x,y)\in E$, and consider the reflexive relations $T$ and $R$ on $E$ defined by the following pullbacks
	$$
	\xymatrix{
		T \ar@{ >->}[d]_-{\langle t_1,t_2\rangle} \ar[r] \pullback & E \ar@{ >->}[d]^-{\langle e_1,e_2\rangle} \\
		E\times E \ar[r]_-{e_1\times e_2} & X\times X }\hspace{5pt}\mathrm{and}\hspace{5pt}
	\xymatrix{
		R \ar[rr] \ar@{ >->}[d]_-{\langle r_1,r_2\rangle} \pullback & & E \ar@{ >->}[d]^-{\langle e_1,e_2\rangle} \\
		E\times E \ar[r]_-{e_2\times e_1} & X\times X \ar[r]^-{\cong}_-{\langle \pi_2,\pi_1\rangle} & X\times X, }
	$$
	where $\pi_1 \colon X \times X \rightarrow X$ and $\pi_2 \colon X \times X \rightarrow X$ are the product projections. We have $(aEb,cEd)\in T$ if and only if $(a,d)\in E$, and $(aEb, cEd)\in R$  if and only if $(c,b)\in E$.
	
	The third reflexive relation on $E$ we consider is the \emph{kernel pair} $\Eq(e_2)$ of $e_2$, defined as the following pullback
	$$\xymatrix{{\Eq(e_2)}   \ar[r] \ar[d] \pullback& E \ar[d]^{e_2} \\
		E \ar[r]_{e_2} & X.}
	$$
	$\Eq(e_2)$ is an equivalence relation,  with the property that $\Eq(e_2)\leqslant R$ and $\Eq(e_2)\leqslant T$, so that $R\wedge \Eq(e_2)=\Eq(e_2)\leqslant T$. We may apply the assumption to the following relations given in solid lines
	$$
	\xymatrix@C=35pt{
		xEy \ar@{-}[r]^-{\Eq(e_2)} \ar@{-}[d]^-R \ar@(l,l)@{-}[d]_-T & yEy \ar@{-}[d]_-R \ar@(r,r)@{--}[d]^-T \\
		xEx \ar@{-}[r]_-{\Eq(e_2)}  & xEx }
	$$
	($xEx$ and $yEy$ by the reflexivity of the relation $E$). We conclude that $(yEy,xEx)\in T$ and, consequently, that $(y,x)\in E$.
\end{proof}

In the proof of the implication (ii) $\Rightarrow$ (i) we only used two ``true'' reflexive relations $R$ and $T$. This observation gives:

\begin{corollary} \label{Mal'tsev <=> SL 2}
	Let $\C$ be a regular category. The following conditions are equivalent:
	\begin{enumerate}
		\item[(i)] $\C$ is a Mal'tsev category;
		%\item[(ii)] The Shifting Lemma holds in $\C$ when $R,S$ and $T$ are reflexive relations;
		\item[(ii)] The Shifting Lemma holds in $\C$ when $R$ and $T$ are reflexive relations and $S$ is an equivalence relation.
	\end{enumerate}
\end{corollary}
\begin{example}
Let $\mathbb T$ be the algebraic theory of a Mal'tsev variety, and ${\mathbb T}(\mathsf{Top})$ the category of \emph{topological Mal'tsev algebras}. This category is regular, essentially because the regular epimorphisms are the \emph{open} surjective homomorphisms \cite{JP}, that are stable under pullbacks. The category ${\mathbb T}(\mathsf{Top})$ is also a Mal'tsev category, so that ${\mathbb T}(\mathsf{Top})$ satisfies the Shifting Lemma for reflexive relations (by Theorem \ref{Mal'tsev <=> SL}). The same is true for the exact Mal'tsev category ${\mathbb T}(\mathsf{Comp})$ of \emph{compact Hausdorff Mal'tsev algebras}. 
\end{example}

%%%%%%%%%%%%%%%%%%%%%%%%%%%%%%%%%%%%%%%%%%%%%%%%%%%%%%%%%%%%%%%%%%%%%%%%%%%%%%%%%%%%%%%%%%%%%%%%%%%%%%%%%%%%%%%%%%%%%%%
\section{Goursat categories and the Shifting Lemma}

A regular category $\C$ is called a \emph{Goursat} category \cite{CKP}  when it is $3$-permuta\-ble, i.e. for any pair of equivalence relations $R$ and $S$ on the same object in $\C$ one has $$RSR=SRS.$$

\begin{remark}\label{join} As already observed in \cite{CLP}, for any pair of equivalence relations $R$ and $S$ on the same object $X$ in a Goursat category, one has that $RSR$ is an equivalence relation, that is then the supremum $R \vee S$ of $R$ and $S$ as equivalence relations on $X$
	$$R\vee S=RSR.$$
\end{remark}

The $3$-permutable version of Theorem~\ref{Mal'tsev chars} is:

\begin{theorem}\label{Goursat chars} \emph{\cite{CKP}}
	Let $\C$ be a regular category. The following conditions are equivalent:
	\begin{enumerate}
		\item[(i)] $\C$ is a Goursat category;
		\item[(ii)] for any relation $\mono{D}{X\times Y}$ in $\C$, $DD^\circ DD^\circ=DD^\circ$;
		\item[(iii)] for any reflexive relation $E$ in $\C$, $EE^\circ$ is an equivalence relation;
		\item[(iv)] for any reflexive relation $E$ in $\C$, $EE^\circ=E^\circ E$.
	\end{enumerate}
\end{theorem}

To adapt our approach from Mal'tsev categories to Goursat categories, we need a suitable property on relations guaranteeing that they are equivalence relations, where this property should be useful to characterize for Goursat categories. A first attempt would be to use the fact that, in a Goursat context, all reflexive and transitive relations are equivalence relations~\cite{M-FRVdL}. However, such a property on relations holds for any $n$-permutable category, $n\geqslant 2$, so it would not provide the needed characterization. The right property we were looking for was discovered after reading~\cite{Tull}, and is given in the next proposition:

\begin{proposition}\label{Goursat positive}
	A regular category $\C$ is a Goursat category if and only if any reflexive and positive relation in $\C$ is an equivalence relation.
\end{proposition}
\begin{proof}
	Suppose that $\C$ is a Goursat category and consider a reflexive and positive relation $P$ with $1_X\leqslant P=U^\circ U$. The fact that $P$ is positive implies that it is symmetric. As for the transitivity of $P$, we have $PP=U^\circ UU^\circ U=U^\circ U$, by Theorem~\ref{Goursat chars} (i) $\Rightarrow$ (ii). It follows that $PP=P$.
	
	Conversely, let $E$ be a reflexive relation on $X$. Then $EE^\circ$ is a reflexive and positive relation, thus an equivalence relation by assumption. It follows that $\C$ is a Goursat category by Theorem~\ref{Goursat chars} (iii) $\Rightarrow$(i).
\end{proof}

If $\C$ is a Goursat category, then the lattice of equivalence relations on the same object is modular~\cite{CKP} and, consequently, the Shifting Lemma holds. Moreover, the Shifting Lemma still holds when $S$ is just a reflexive relation, as we show next. The following result is partly based on Lemma 2.2 in~\cite{Kiss}, and it gives a first step towards the characterization we aim to obtain for Goursat categories.

\begin{proposition}
	\label{Goursat => SL}
	In any regular Goursat category $\C$, the Shifting Lemma holds when $S$ is a reflexive relation and $R$ and $T$ are equivalence relations.
\end{proposition}
\begin{proof}
	Let $R$ and $T$ be equivalence relations and let $S$ be a reflexive relation on an object $X$ such that $R\wedge S\leqslant T$. Suppose that we have $(x,y)\in R \wedge T$, $(x,u) \in S$, $(y,v)\in S$ and $(u,v)\in R$ as in \eqref{SL}. We consider the two equivalence relations on $S$ determined by the following pullbacks
	$$
	\xymatrix@C=55pt{
		R\Box S \ar@{ >->}[d]_-{\langle \pi_{13},\pi_{24}\rangle} \ar@{ >->}[r]^-{\langle \pi_{12},\pi_{34}\rangle} \pullback &
		S\times S \ar@{ >->}[d]^-{\langle s_1,s_2\rangle\times \langle s_1,s_2\rangle} \\
		R\times R \ar@{ >->}[r]_-{\langle r_1\times r_1,r_2\times r_2\rangle} & X\times X\times X\times X} \hspace{50pt}
	$$
	and
	$$
	\xymatrix@C=55pt{
		W \ar@{ >->}[d]_-{\langle \pi_{13},\pi_{24}\rangle} \ar@{ >->}[r]^-{\langle \pi_{12},\pi_{34}\rangle} \pullback &
		S\times S \ar@{ >->}[d]^-{\langle s_1,s_2\rangle\times \langle s_1,s_2\rangle} \\
		T\times R \ar@{ >->}[r]_-{\langle t_1\times r_1,t_2\times r_2\rangle} & X\times X\times X\times X.} \hspace{50pt}
	$$
	We have $(aSb, cSd)\in R\Box S$ if and only if
	$$
	\begin{array}{ccc}
	a & S & b \\
	R & & R \\
	c & S & d
	\end{array}
	$$
	and $(aSb, cSd)\in W$ if and only if
	$$
	\begin{array}{ccc}
	a & S & b \\
	T & & R \\
	c & S & d.
	\end{array}
	$$
	Note that they are in fact equivalence relations on $S$ since $R$ and $T$ are both equivalence relations.
	
	Given the equivalence relations $R\Box S$, $\Eq(s_2)$ and $W$ on $S$, Remark~\ref{join} yields the following description of the supremum of $R\Box S\wedge \Eq(s_2)$ and $W$ as equivalence relations on $S$:
	$$
	\begin{array}{lcl}
	(\,R\Box S\wedge \Eq(s_2)\,)\, \vee\, W & = & (\,R\Box S\wedge \Eq(s_2)\,)\,W\,(\, R\Box S\wedge \Eq(s_2)\,) \vspace{3pt}\\
	& = & W\,(\,R\Box S\wedge \Eq(s_2)\,)\,{W.}
	\end{array}
	$$
	Since $$R\Box S \wedge \Eq(s_2)\leqslant (\,R\Box S\wedge \Eq(s_2)\,)\,\vee\, W$$ we may apply the Shifting Lemma to the following diagram
	$$
	\xymatrix@C=45pt{
		xSu \ar@{-}[r]^-{\Eq(s_2)} \ar@{-}[d]^-{R\Box S} \ar@(l,l)@{-}[d]_-{(\, R\Box S\wedge \Eq(s_2)\,)\, \vee\,W} &
		uSu \ar@{-}[d]_-{R\Box S} \ar@(r,r)@{--}[d]^-{(\, R\Box S\wedge \Eq(s_2)\,)\, \vee\,W} \\
		ySv \ar@{-}[r]_-{\Eq(s_2)} & vSv.}
	$$
Note that, $uSu$ and $vSv$ by the reflexivity of $S$. We then obtain $$(uSu, vSv)\in (\, R\Box S\wedge \Eq(s_2)\,)\, \vee\,W.$$ Using $$(\, R\Box S\wedge \Eq(s_2)\,)\, \vee\,W=(\,R\Box S\wedge \Eq(s_2)\,)\,W\,(\, R\Box S\wedge \Eq(s_2)\,)$$ this means that
	$$
	(uSu) \left( R\Box S \wedge \Eq(s_2) \right) (aSu) W (bSv) \left( R\Box S \wedge \Eq(s_2) \right) (vSv),
	$$
	for some $a,b$ in $X$, i.e.
	$$
	\begin{array}{ccc}
	u & S & u\\
	R & & R \\
	a & S & u \\
	T & & R \\
	b & S & v \\
	R & & R \\
	v & S & v.
	\end{array}
	$$
	Since $aRu$ ($R$ is symmetric), $aSu$ and $R\wedge S\leqslant T$, it follows that $aT u$; similarly $bTv$. From $uTa$ ($T$ is symmetric), $aTb$ and $bTv$, we
	conclude that $uTv$ ($T$ is transitive), as desired.
\end{proof}

\begin{remark}
	The Shifting Lemma when $S$ is a reflexive relation and $R$ and $T$ are equivalence relations, as stated in Proposition~\ref{Goursat => SL}, is the categorical version of the {Shifting Principle} recalled in the Introduction. First, assuming that $R\wedge S\leqslant T$ is equivalent to assuming that $R\wedge S\leqslant T\leqslant R$ for the property of diagram \eqref{SL} to hold (take $R\wedge T$, which is such that $R\wedge S\leqslant R\wedge T\leqslant R$, for the non-obvious implication). Second, going carefully through the proofs in~\cite{Gumm}, one may check that the symmetry of $S$ is not necessary. So the {Shifting Principle} could equivalently be stated by asking that $S$ is just a reflexive and compatible relation.
\end{remark}

We now use Propositions~\ref{Goursat positive} and~\ref{Goursat => SL} to obtain the characterization of Goursat categories through a variation of the Shifting Lemma:

\begin{theorem}
	\label{Goursat <=> SL}
	Let $\C$ be a regular category. The following conditions are equivalent:
	\begin{enumerate}
		\item[(i)] $\C$ is a Goursat category;
		\item[(ii)] The Shifting Lemma holds in $\C$ when $S$ is a reflexive relation and $R$ and $T$ are reflexive and positive relations.
	\end{enumerate}
\end{theorem}
\begin{proof}
	(i) $\Rightarrow$ (ii) This implication follows from the fact that reflexive and positive relations are necessarily equivalence relations in the Goursat context (Proposition~\ref{Goursat positive}) and from Proposition~\ref{Goursat => SL}. \\
	(ii) $\Rightarrow$ (i) We shall prove that for any reflexive relation $E$ on $X$ in $\C$, $EE^\circ= E^\circ E$ (see Theorem~\ref{Goursat chars} (iv)). Suppose that $(x,y)\in EE^\circ$. Then, for some $z$ in $X$, one has that $(z,x)\in E$ and $(z,y)\in E$. Consider the reflexive and positive relations $R=EE^\circ$ and $T=E^\circ E$, and the reflexive relation $E$ on $X$. From the reflexivity of $E$, we get $E\leqslant EE^\circ$ and $E\leqslant E^\circ E$; thus $EE^\circ \wedge E=E\leqslant E^\circ E$. We may apply our assumption to the following relations given in solid lines:
	\begin{equation}\label{SLR}
	\vcenter{\xymatrix@C=45pt{
			z \ar@{-}[r]^-E \ar@{-}[d]^-{EE^\circ} \ar@(l,l)@{-}[d]_-{E^\circ E} & x \ar@{-}[d]_-{EE^\circ} \ar@(r,r)@{--}[d]^-{E^\circ E} \\
			z \ar@{-}[r]_-E  & y }}
		\end{equation}	
	to conclude that $(x,y)\in E^\circ E$. Having proved that $EE^\circ \leqslant E^\circ E$ for \emph{every} reflexive relation $E$, the equality $E^\circ E\leqslant EE^\circ$ follows immediately.
\end{proof}
In the proof of the implication (ii) $\Rightarrow$ (i) we  used the relations $EE^o$ and $E^oE$ with $E$ a reflexive relation. This observation gives:
	\begin{corollary}
		Let $\C$ be a regular category. The following conditions are equivalent:
		\begin{enumerate}
			\item[(i)] $\C$ is a Goursat category;
			\item[(ii)] The Shifting Lemma holds in $\C$ for a diagram as \eqref{SLR}, with $E$ any reflexive relation.
		\end{enumerate}
\end{corollary}
The fact that any quasi-variety is a regular category \cite{PV} implies the following 
\begin{corollary}
	Let $\V$ be a quasi-variety. The following conditions are equivalent:
	\begin{enumerate}
		\item[(i)] $\V$ is $3$-permutable;
		\item[(ii)] The Shifting Lemma holds in $\V$ when $S$ is a reflexive compatible relation, and $R$ and $T$ are reflexive and positive compatible relations.
	\end{enumerate}
\end{corollary}

\end{document}